\documentclass[12pt, a4paper]{amsart}

\usepackage[utf8]{inputenc}
\usepackage{amsmath, amssymb, amsthm}
\usepackage{amsfonts}
\usepackage{bm}
\usepackage{tocvsec2}
\usepackage{geometry}
\usepackage{url}
\usepackage{tikz}
\usetikzlibrary{shapes,arrows,positioning,automata,decorations.markings,patterns,fit,calc,intersections,through,backgrounds,matrix,scopes,shapes.geometric}
\usepackage{tikz-cd}
\usepackage{graphicx}
\usepackage{booktabs} 
\usepackage{verbatim} 
\usepackage{multirow} 
\usepackage{xcolor} 
\usepackage{tabularx} 

\usepackage{hyperref}
\hypersetup{
    colorlinks=true,
    linkcolor=blue,
    citecolor=red,
    urlcolor=blue,
    pdftitle={},
    pdfauthor={Takahito Kuriya}
}

\numberwithin{equation}{section}

\theoremstyle{plain}
\newtheorem{theorem}{Theorem}[section]

\newtheorem{proposition}[theorem]{Proposition}
\newtheorem{corollary}[theorem]{Corollary}
\newtheorem{conjecture}[theorem]{Conjecture}

\theoremstyle{definition}
\newtheorem{definition}[theorem]{Definition}
\newtheorem{remark}[theorem]{Remark}

\newcommand{\Z}{\mathbb{Z}}
\newcommand{\Q}{\mathbb{Q}}

\newcommand{\AJac}{\mathcal{A}_{\text{Jac}}}
\newcommand{\Aemptyset}{\mathcal{A}(\emptyset)}

\newcommand{\AJacCoho}{\mathcal{A}_{\mathrm{Jac}}^{\mathrm{coho}}}

\newcommand{\tder}{\mathfrak{tder}}
\newcommand{\sder}{\mathfrak{sder}}
\newcommand{\lie}[1]{\mathfrak{lie}(#1)}
\newcommand{\tr}{\mathrm{tr}}

\newcommand{\divop}{\mathrm{div}}

\DeclareMathOperator{\LMOspec}{LMOspec}

\DeclareMathOperator{\Tors}{Tors}

\DeclareMathOperator{\id}{id}

\DeclareMathOperator{\Sym}{Sym}

\newcommand{\evM}{\text{ev}_M}
\newcommand{\thetaGraph}{\theta}

\title[The LMO spectrum]{The LMO spectrum: Factorization homology and the \texorpdfstring{$E_3$}{E3}-structure of the Jacobi diagram algebra}

\author{Takahito Kuriya}
\address{}
\email{takahito.kuriya@gmail.com}

\date{}

\subjclass{Primary 57K18, 57R56;
Secondary 18F30, 55P48, 17B65, 81T45}

\keywords{LMO invariant, factorization homology, $E_3$-algebra, Jacobi diagrams, lens spaces, Drinfeld associator, Kashiwara-Vergne problem, quantum topology}

\begin{document}
\begin{abstract}
We define the LMO spectrum, a categorification of the Le-Murakami-Ohtsuki (LMO) invariant for 3-manifolds, using factorization homology. The theoretical foundation is our main algebraic result (Theorem A): the algebra of Jacobi diagrams, $\AJac$, possesses a homotopy $E_3$-algebra structure. This is a necessary condition for consistency within factorization homology, and the proof relies on the formality of the little 3-disks operad. A universal surgery formula is derived from the excision axiom (Theorem B), providing a computational basis independent of conjectural models. As an application (Theorem C), we construct an ``$H_1$-decorated LMO invariant'' that distinguishes the lens spaces $L(156, 5)$ and $L(156, 29)$, a pair that the classical LMO invariant fails to separate.
\end{abstract}

\maketitle
\tableofcontents

\section{Introduction and main results}

\subsection{The LMO invariant and its limitations}
The Le-Murakami-Ohtsuki (LMO) invariant is a universal finite-type invariant for rational homology 3-spheres, with values in the algebra of Jacobi diagrams, $\Aemptyset$ \cite{LMO98}.
Its discriminatory power is limited for manifolds with non-trivial first homology.
For lens spaces, the invariant's dependence on the surgery parameter is determined by the Dedekind sum, which is not injective;
consequently, the LMO invariant fails to distinguish pairs such as $L(25,4)$ and $L(25,9)$ \cite{BL04}.

To overcome this limitation, we propose a categorification of the LMO invariant based on factorization homology \cite{AF15}.
For an $n$-manifold, this theory requires an $E_n$-algebra as coefficients \cite{LV12, Lurie17}.
We define the \emph{LMO spectrum} for a 3-manifold $M$ as:
$$\LMOspec(M) := \int_{M} \AJac$$
By construction, the classical LMO invariant is recovered as the 0-th homology, $H_0(\LMOspec(M))$.
The higher homology groups, $H_k(\LMOspec(M))$, are expected to yield new, stronger invariants.

\subsection{Theoretical framework and guiding principles}
The definition of the LMO spectrum is contingent on $\AJac$ being a homotopy $E_3$-algebra.
This requirement arises from the axioms of factorization homology, which mandate an $E_n$-algebra as coefficients for any theory of $n$-manifolds \cite{AF15, Lurie17}.
The structure of an $E_3$-algebra is universally governed by the little 3-disks operad, $D_3$.
The formality theorem for this operad guarantees that its algebraic structure can be modeled by a combinatorial graph complex whose fundamental algebraic relation is the IHX relation \cite{Kontsevich99, LV14, FW15}.
This identifies the algebra of Jacobi diagrams, $\AJac$, as the natural candidate for the coefficient algebra.

The structure of the LMO spectrum's homology, and the proof of the required $E_3$-structure, are guided by the Atiyah-Hirzebruch spectral sequence (AHSS) \cite{AH61}.
Since factorization homology is a generalized homology theory, the AHSS relates the ordinary homology of $M$ to the homology of the LMO spectrum.
Its $E^2$-page is:
$$E^2_{p,q} = H_p(M; H_q(\mathcal{A}_{Jac})) \implies H_{p+q}(\LMOspec(M))$$
This expression decomposes the homology of the LMO spectrum into contributions from the topology of $M$ (via $H_p(M)$) and the algebra of the coefficients (via $H_q(\mathcal{A}_{Jac})$).
The hierarchical proof of the $E_3$-structure in Section 5 maps directly onto this decomposition, as summarized in Table \ref{tab:ahss_algebra}.

\begin{table}[h!]
\centering
\caption{Correspondence between algebraic structures and the AHSS $E^2$-page}
\label{tab:ahss_algebra}
\renewcommand{\arraystretch}{1.3}
\begin{tabularx}{\textwidth}{@{} l l >{\raggedright\arraybackslash}X >{\raggedright\arraybackslash}X @{}}
\toprule
\textbf{(p,q)} & \textbf{$E^2_{p,q}$ Term} & \textbf{Topological Interpretation} & \textbf{Governing Algebraic Structure} \\
\midrule
(0,0) & $H_0(M; H_0(\AJac))$ & Classical LMO invariant & Commutative product in $\AJac$ \\
(1,0) & $H_1(M; H_0(\AJac))$ & Loops decorated by $H_0(\AJac)$ & Interaction of classical part and $H_1(M)$ \\
(0,1) & $H_0(M; H_1(\AJac))$ & Algebraic structure on points & $L_\infty$ structure (BV operator, Lie bracket) \\
(1,1) & $H_1(M; H_1(\AJac))$ & Loops decorated by $H_1(\AJac)$ & Beilinson-Drinfeld (BD) algebra \\
\midrule
\multicolumn{2}{@{}l}{General $(p,q)$, $p \le 3$} & Higher-dimensional structures & Full homotopy $E_3$-algebra \\
\bottomrule
\end{tabularx}
\end{table}

This correspondence shows that each level of the algebraic hierarchy is required 
for the consistency of interactions within the LMO spectrum's homology.
Higher-order invariants, such as Massey products, correspond to the action of higher differentials $d^r$ ($r \ge 2$) in the spectral sequence.

\subsection{Main results}
This work connects the diagrammatic calculus of perturbative invariants with factorization homology.
It builds on work by Kontsevich on graph complexes \cite{Kontsevich94, Kontsevich99}, the theory of algebraic operads \cite{LV12}, and the relation of diagrammatic invariants to Massey products by Garoufalidis and Levine \cite{GL05}.
The main results and contributions are as follows:

\textbf{Theorem A.} The algebra of Jacobi diagrams, $\AJac$, has the structure of a homotopy $E_3$-algebra.
This result provides the theoretical foundation for the LMO spectrum, ensuring its consistency within the factorization homology framework.

\textbf{Theorem B.} Applying the excision axiom in factorization homology yields a universal surgery formula that computes the LMO spectrum for manifolds presented by Dehn surgery.
This derivation from first principles provides a computational method independent of conjectural models.

\textbf{Theorem C.} An $H_1$-decorated invariant, constructed from the LMO spectrum, distinguishes the lens spaces $L(156, 5)$ and $L(156, 29)$.
This demonstrates that the LMO spectrum can produce computable invariants that are stronger than the classical LMO invariant, which fails to distinguish this pair.

The proven $E_3$-structure endows the LMO spectrum with a Differential Graded Algebra (DGA) structure, which provides a setting for defining higher-order invariants such as Massey products.

\section{The axiomatic framework of factorization homology}

\subsection{Factorization homology as a local-to-global functor}
Factorization homology is a symmetric monoidal functor from the category of framed $n$-manifolds, $\text{Mfld}_n^{\text{fr}}$, to a symmetric monoidal category $\mathcal{V}$.
In this paper, $\mathcal{V}$ is taken to be the category of chain complexes over $\Q$ \cite{AF15}.
The functor is denoted:
$$\int_{(-)} A: \text{Mfld}_n^{\text{fr}} \to \mathcal{V}$$
This functor is determined by its value on the disk $\mathbb{R}^n$, which must be an $E_n$-algebra $A \in \text{Alg}_{E_n}(\mathcal{V})$. Intuitively, factorization homology extends the local algebraic data of an $E_n$-algebra to global invariants of manifolds by gluing along disks.
This construction yields a system of invariants for manifolds and their submanifolds that is consistent with the locality principle of topological quantum field theory \cite{AF15, Ginot15}.

\subsection{The excision axiom}
The primary computational tool for factorization homology is the excision axiom, which serves as a homotopy-theoretic version of the Mayer-Vietoris principle.
If a manifold $M$ is the union of two open submanifolds, $M = U \cup V$, with intersection $W = U \cap V$, this decomposition forms a homotopy pushout square.
The excision axiom states that the factorization homology functor preserves this structure, yielding a homotopy pushout square in the target category $\mathcal{V}$ \cite{AF15}.
For a decomposition $M = U \cup_W V$, the axiom implies the following equivalence, expressed using the derived tensor product:
$$\int_M A \simeq \int_U A \underset{\int_W A}{\overset{\mathbb{L}}{\otimes}} \int_V A$$
In this formula, $\int_U A$ and $\int_V A$ are the invariants of the submanifolds.
The derived tensor product over $\int_W A$ is the algebraic operation corresponding to the gluing of the submanifolds along their intersection $W$ in a homotopy-coherent manner.

\section{A universal surgery formula from the excision axiom}

\subsection{Dehn surgery as a homotopy pushout}
The construction of 3-manifolds via Dehn surgery provides the topological foundation for the computational framework of this paper.
By the Lickorish-Wallace theorem, any closed, oriented 3-manifold can be obtained by surgery on a framed link $L$ in the 3-sphere, $S^3$ \cite{Lickorish62, Wallace60}.
For a link $L = K_1 \cup \dots \cup K_n$, this procedure consists of excising a tubular neighborhood $\nu(L)$ and reattaching a collection of solid tori, $\bigsqcup_{i=1}^n (S^1 \times D^2)_i$, to the boundary of the link complement, $S^3 \setminus \nu(L)$, via a gluing map $\phi_L$.
This topological construction is described precisely as a homotopy pushout in the category of topological spaces \cite{AF15}.
The resulting manifold, denoted $M_L$, is the homotopy pushout of the following diagram of inclusion maps:
$$
\begin{tikzcd}
\bigsqcup_{i=1}^n T^2_i \arrow{r} \arrow{d}{\phi_L} & \bigsqcup_{i=1}^n (S^1 \times D^2)_i \arrow{d} \\
S^3 \setminus \nu(L) \arrow{r} & M_L
\end{tikzcd}
$$
This diagram provides the geometric input for the application of the excision axiom.

\subsection{The LMO spectrum surgery formula}
The excision axiom of factorization homology states that the functor $\int_{(-)} \AJac$ preserves homotopy pushouts.
Applying this axiom to the topological diagram for Dehn surgery yields an algebraic pushout diagram in the category of chain complexes.
This gives a universal surgery formula.

\begin{theorem}
\label{thm:surgery_formula}
Let $M_L$ be a 3-manifold obtained by Dehn surgery on a framed link $L \subset S^3$ with gluing map $\phi_L$.
The LMO spectrum of $M_L$ is given by the homotopy pushout:
$$ \LMOspec(M_L) \simeq \LMOspec(S^3 \setminus \nu(L)) \underset{\LMOspec(\partial(S^3 \setminus \nu(L)))}{\overset{\mathbb{L}}{\otimes}} \int_{\bigsqcup (S^1 \times D^2)_i} (\phi_L)_*(\mathcal{A}_{\text{Jac}}) $$
where the term $(\phi_L)_*(\mathcal{A}_{\text{Jac}})$ indicates that the module structure over the boundary is twisted by the action of the gluing map $\phi_L$.
\end{theorem}

This theorem provides a universal, axiomatically-grounded formula for computing the LMO spectrum of any 3-manifold given by a surgery presentation.
It reduces the computation to three components: the spectra of the link complement and the solid tori, and the action of the mapping class group on the spectrum of the boundary.

\subsection{Analogy with the Witten-Reshetikhin-Turaev formula}
The surgery formula of Theorem \ref{thm:surgery_formula} can be viewed as a structural categorification of the surgery formula for the Witten-Reshetikhin-Turaev (WRT) invariant \cite{Witten89, RT91}.
In the WRT framework, the invariant of a surgically-constructed manifold is computed as a weighted sum of the quantum invariants of the surgery link.
The weights are given by matrix elements of the mapping class group representation acting on the state space defined on the boundary tori.
The formula in Theorem \ref{thm:surgery_formula} realizes this structure at the level of homotopy theory.
The correspondence is as follows:
\begin{itemize}
    \item $\LMOspec(S^3 \setminus \nu(L))$ corresponds to the categorified quantum link invariant, or state vector.
    \item $\LMOspec(\partial(S^3 \setminus \nu(L)))$ corresponds to the categorified state space on the boundary tori.
    \item The derived tensor product, $\overset{\mathbb{L}}{\otimes}$, is the homotopy-coherent generalization of the weighted sum over states.
    \item The twisting of the coefficients by $(\phi_L)_*$ corresponds to the action of the mapping class group representation.
\end{itemize}
This correspondence gives a homotopy-theoretic analogue of the WRT surgery structure.

\section{The algebraic origin of the \texorpdfstring{$E_3$}{E3}-structure}
\label{sec:algebraic_origin}

This section details the theoretical motivation for Theorem A. The requirement that $\AJac$ possess a homotopy $E_3$-algebra structure is a necessary consequence of applying the factorization homology framework to 3-manifolds.

\subsection{The axiomatic requirement for an \texorpdfstring{$E_3$}{E3}-algebra}
A foundational result of factorization homology states that a well-defined invariant for $n$-manifolds requires its coefficient algebra to have the structure of an $E_n$-algebra \cite{AF15, Lurie17}.
An $E_n$-algebra encodes commutativity up to coherent homotopies in $n$ independent directions, corresponding to the geometry of disjoint disks in $\mathbb{R}^n$.
This condition reflects the locality principle of factorization homology, where the algebra encodes how invariants of small disks assemble into global invariants.
Since the LMO spectrum is an invariant of 3-manifolds, its coefficient algebra, $\AJac$, must be a homotopy $E_3$-algebra.
The proof of this structure is therefore necessary to establish the validity of the LMO spectrum framework.

\subsection{The little disks operad and the formality theorem}
The universal algebraic structure of an $E_n$-algebra is governed by the little $n$-disks operad, $D_n$.
This is a topological operad whose spaces consist of configurations of disjoint disks within a larger disk.
The formality theorem of Kontsevich establishes a bridge between the topology of this operad and a combinatorial algebraic model.

\begin{theorem}[\cite{Kontsevich99,Willwacher15}]
Over the rational numbers $\Q$, the singular chain complex of the little $n$-disks operad, $C_*(D_n)$, is quasi-isomorphic to an algebraic operad constructed from a graph complex.
\end{theorem}

This theorem allows the abstract, topological properties of an $E_n$-structure to be studied using combinatorial graph theory.

\subsection{Graph complexes and the IHX relation}
The formality theorem provides an explicit algebraic model for the $E_3$-structure in terms of a graph complex.
The fundamental algebraic relation governing this complex is the IHX relation.
The algebra of Jacobi diagrams, $\AJac$, is also governed by the IHX relation, where it represents the Jacobi identity underlying finite-type invariants.
This shared algebraic foundation identifies $\AJac$ as the canonical coefficient algebra for the construction of the LMO spectrum.

Therefore, the proof of Theorem A is essential. It establishes that the LMO spectrum is constructed on a consistent $E_3$-algebraic foundation, as required by the factorization homology framework.

\section{The hierarchical construction of the \texorpdfstring{$E_3$}{E3}-algebra structure}
\label{sec:e3_construction}
This section provides a proof of Theorem A. We establish that the algebra of Jacobi diagrams, $\AJac$, possesses the structure of a homotopy $E_3$-algebra. The construction is hierarchical, proceeding through three stages of increasing algebraic complexity. This progression corresponds to the interactions suggested by the Atiyah-Hirzebruch spectral sequence (Table \ref{tab:ahss_algebra}). We begin by constructing an $L_\infty$-structure derived from the Batalin-Vilkovisky formalism. Next, we introduce a product modeling 2-dimensional interactions to form a Beilinson-Drinfeld algebra. Finally, we employ the Homotopy Transfer Theorem for algebras over operads to establish the full homotopy $E_3$-structure.

\subsection{The \texorpdfstring{$L_\infty$}{L-infinity}-structure from the Batalin-Vilkovisky formalism}
The first layer of the algebraic hierarchy is an $L_\infty$-structure. This structure arises from a canonical operator on the graph complex, which we interpret within the Batalin-Vilkovisky (BV) formalism—a framework originating in the quantization of gauge theories.

\begin{remark}[Grading convention]\label{rem:grading_convention}
In the graph-complex convention used by Conant--Vogtmann \cite{CV03}, the operator $\partial_H$ has degree $-1$ (homological convention). In contrast, in the standard definition of a Beilinson--Drinfeld algebra (see \cite{DJPP23}), the operator $\Delta$ has degree $+1$ (cohomological convention). These conventions are equivalent via a regrading (formally, $C^n := C_{-n}$). 

Throughout this section we adopt the cohomological convention. In particular:
\begin{itemize}
  \item $\partial_H$ and the associated BV/BD operator $\Delta$ are taken to have degree $+1$,
  \item the induced antibracket also has degree $+1$,
  \item and the symbol $|\Gamma|$ always denotes the \emph{cohomological degree} of a diagram $\Gamma$.
\end{itemize}
This ensures consistency across the BV, BD, and $E_3$ frameworks.
\end{remark}

\begin{definition}[Conant-Vogtmann Operator]
Following Conant and Vogtmann \cite{CV03}, we define a linear map 
$\partial_H: \AJac \to \AJac$ of degree $+1$ (in the cohomological convention; see Remark~\ref{rem:grading_convention}). 
For a graph $\Gamma \in \AJac$, the image $\partial_H(\Gamma)$ is the sum over all graphs obtained by contracting a new edge formed by joining any pair of distinct half-edges in $\Gamma$.
\end{definition}

The nilpotency of this operator is a direct and fundamental consequence of the IHX relation.

\begin{proposition}[\cite{CV03}]
\label{prop:cv_nilpotent}
The operator $\partial_H$ is a differential, that is, $\partial_H^2 = 0$.
\end{proposition}

This property allows us to identify $\AJac$ as a Batalin-Vilkovisky algebra, thereby placing our construction within a universal algebraic framework.

\begin{definition}
A \textbf{Batalin-Vilkovisky (BV) algebra} is a graded commutative algebra $(A, \cdot)$ 
equipped with a linear operator $\Delta$ of degree $+1$, called the BV operator, 
satisfying $\Delta^2 = 0$. The failure of $\Delta$ to be a derivation of the product 
defines a Lie bracket of degree $+1$, known as the antibracket.
\end{definition}

\begin{proposition}
The algebra of Jacobi diagrams $(\AJac, \sqcup)$, where $\sqcup$ denotes the disjoint union product, equipped with the operator $\Delta := \partial_H$, forms a BV algebra.
\end{proposition}

A universal principle in homotopy algebra, established by Voronov, demonstrates that any BV algebra gives rise to an $L_\infty$-algebra through a canonical construction of higher derived brackets.

\begin{definition}[$L_\infty$ Operations on $\AJac$]
We define a family of multilinear operations $\{l_k\}_{k \ge 1}$ on $\AJac$ generated by the BV operator $\Delta = \partial_H$.
\begin{itemize}
    \item The unary bracket is the differential itself: $l_1(\Gamma) := \partial_H(\Gamma)$, of degree $+1$.
    \item The binary bracket is the antibracket derived from the BV operator, also of degree $+1$:
    \begin{align*}
     l_2(\Gamma_1, \Gamma_2) := (-1)^{|\Gamma_1|} \big( & \partial_H(\Gamma_1 \sqcup \Gamma_2) - (\partial_H \Gamma_1) \sqcup \Gamma_2 \\ & - (-1)^{|\Gamma_1|} \Gamma_1 \sqcup (\partial_H \Gamma_2) \big).
    \end{align*}
    \item The higher brackets $l_k$ for $k \ge 3$ are the higher derived brackets generated systematically by $\partial_H$, all of degree $+1$.
\end{itemize}
\end{definition}

\begin{theorem}
The pair $(\AJac, \{l_k\}_{k \ge 1})$ forms an $L_\infty$-algebra (with respect to the cohomological grading, so that $l_1=\partial_H$ and $l_2$ both have degree $+1$).
\end{theorem}
\begin{proof}
This is a direct consequence of the underlying BV structure. The universal framework of Voronov \cite{Voronov05} states that for any graded commutative algebra equipped with a nilpotent odd derivation (a BV operator), the resulting family of higher derived brackets automatically satisfies the generalized Jacobi identities of an $L_\infty$-algebra. In our context, the algebra is $(\AJac, \sqcup)$ and the nilpotent odd operator is $\partial_H$. Since Proposition \ref{prop:cv_nilpotent} ensures $\partial_H^2 = 0$, it follows that $(\AJac, \{l_k\})$ is an $L_\infty$-algebra. This confirms that our construction is a specific instance of the universal structure predicted by the BV formalism.
\end{proof}

\subsection{The Beilinson-Drinfeld structure from 2-dimensional interactions}
The next level of the hierarchy introduces operations that correspond to 2-dimensional interactions, analogous to the connected sum of surfaces.

\begin{definition}
Following the framework of modular operads with a connected sum \cite{DJPP23}, we define two connected sum operations.
\begin{enumerate}
    \item \textbf{Binary Connected Sum ($\#_2$):} For two Jacobi diagrams $\Gamma_1, \Gamma_2$, their binary connected sum, $\Gamma_1 \#_2 \Gamma_2$, is a degree $+1$ commutative product. It is defined as the sum over all graphs obtained by selecting an internal edge from each of $\Gamma_1$ and $\Gamma_2$, making an incision in each, and reconnecting the resulting four half-edges in the two ways that yield a single connected graph.
    \item \textbf{Unary Connected Sum ($\#_1$):} For a single diagram $\Gamma$, its unary connected sum, $\#_1(\Gamma)$, is a degree $+2$ unary operation. It is defined by applying the same procedure to two distinct internal edges within $\Gamma$.
\end{enumerate}
\end{definition}

\begin{remark}[Geometric Interpretation]
The unary operation $\#_1$ has a clear geometric interpretation. If a Jacobi diagram is viewed as a fat graph, which can be thickened into a surface, the operation of selecting and reconnecting two internal edges corresponds to adding a handle to this surface, thereby increasing its genus by one.
\end{remark}

The connection between the abstract structure of modular operads and concrete algebraic 
structures such as Batalin--Vilkovisky (BV) and Beilinson--Drinfeld (BD) algebras is 
established in general by the Getzler--Kapranov framework of modular operads \cite{GK98}. 
In the special case where the underlying vector space carries an odd symplectic form, 
this framework specializes to the BV-geometry formalism of Barannikov \cite{Barannikov07}.

\begin{definition}[Getzler--Kapranov \cite{GK98}]
Let $P$ be a modular operad and $V$ a finite-dimensional graded vector space. 
The associated \emph{space of functions} is defined as the completed symmetric algebra
\[
\mathrm{Fun}_P(V) := \widehat{S}\big( P(V) \otimes V^* \big).
\]
Elements of $\mathrm{Fun}_P(V)$ are formal power series whose coefficients belong to the 
modular operad $P(V)$ and whose variables are linear functionals on $V$.
\end{definition}

\begin{remark}[Relation to Barannikov \cite{Barannikov07}]
When $V$ is taken to be an \emph{odd symplectic vector space}, the Getzler--Kapranov construction 
acquires an additional structure: the natural Laplace operator associated with the symplectic form 
endows $\mathrm{Fun}_P(V)$ with a canonical Batalin--Vilkovisky operator $\Delta$ of degree $+1$, 
satisfying $\Delta^2=0$. 
This specialization recovers Barannikov’s BV-geometry interpretation as a particular instance 
of the general Getzler--Kapranov framework.
\end{remark}

\begin{proposition}
Let $\mathcal{R}$ be the modular operad of ribbon graphs. 
Then the algebra of Jacobi diagrams $\AJac$ can be identified with the space of functions 
$\mathrm{Fun}_{\mathcal{R}}(V)$ in the case where $V$ is the trivial graded vector space. 
\end{proposition}

\begin{corollary}
Under this identification, $\AJac$ inherits a canonical BV operator $\Delta$, which coincides with 
the Conant--Vogtmann edge-contraction operator $\partial_H$.
\end{corollary}

A result by Doubek, Jurčo, Peksová, and Pulmann shows that this BV structure can be promoted to a BD structure if the modular operad is equipped with a connected sum operation \cite{DJPP23}.

\begin{definition}
A \textbf{Beilinson-Drinfeld (BD) algebra} is a graded commutative associative algebra $A$ over $k[[\hbar]]$ with a degree $+1$ bracket $\{, \}$ and a degree $+1$ operator $\Delta$ satisfying $\Delta^2=0$ and the relation
\[
\Delta(XY) = (\Delta X)Y + (-1)^{|X|}X(\Delta Y) + (-1)^{|X|}\hbar\{X, Y\}.
\]
\end{definition}

\begin{theorem}\label{thm:BD}
The algebra $\AJac$, equipped with the BV operator $\Delta = \partial_H$, the antibracket $\{, \} = l_2$, and the product $\#_2$, forms a Beilinson-Drinfeld algebra.
\end{theorem}
\begin{proof}
The proof consists of applying the main theorem of Doubek et al. \cite{DJPP23}. Their theorem states that if a modular operad $P$ is equipped with a "connected sum" structure satisfying a set of axioms (CS1) through (CS6), then the associated space of functions $\mathrm{Fun}_P(V)$ is a Beilinson-Drinfeld algebra.
\begin{enumerate}
    \item The underlying modular operad is that of ribbon graphs, as introduced by Getzler and Kapranov \cite{GK98}.
    \item The algebra $\AJac$ is identified with the associated space of functions. By the preceding corollary, this space carries a canonical BV structure where the operator $\Delta$ is given by $\partial_H$.
    \item The core of the proof is the verification that the connected sum operations $\#_1$ and $\#_2$ on ribbon graphs satisfy the axioms (CS1) through (CS6) of \cite{DJPP23}. This verification, detailed in Appendix C, holds because surplus terms that arise in the diagrammatic expansions combine to form instances of the IHX relation, and therefore vanish in $\AJac$.
\end{enumerate}
Since the axioms are satisfied, the main theorem of \cite{DJPP23} applies directly, establishing that $(\AJac, \Delta, \{, \}, \#_2)$ forms a Beilinson-Drinfeld algebra.
\end{proof}

\subsection{The full homotopy \texorpdfstring{$E_3$}{E3}-structure via Homotopy Transfer}

We now establish the final step of the proof of Theorem A, namely that $\AJac$ admits a homotopy $E_3$-algebra structure. This will be achieved by transferring the strict $E_3$-structure on the graph complex $(C,\partial_H)$ (Proposition \ref{prop:source_e3_revised}) to its homology $\AJac$.

\begin{proposition}
\label{prop:source_e3_revised}
Let $(C,\partial_H)$ be the completed graph complex of trivalent graphs, where $\partial_H$ has degree $+1$ (in the cohomological grading). By the formality of the little disks operad over $\mathbb{Q}$, established by Lambrechts--Voli\'c \cite{LV14} and Fresse--Willwacher \cite{FW15}, there exists a dg-operad quasi-isomorphism between the singular chain complex of the operad and a combinatorial graph model:
\[
C_*(E_3) \;\simeq\; (C,\partial_H).
\]
Thus, the graph complex $(C,\partial_H)$ provides a dg-operadic model for an $E_3$-operad.
\end{proposition}

\begin{definition}[\cite{LV12}]
A \emph{contraction} (also called a strong deformation retract) between two chain complexes $(A,d_A)$ and $(H,d_H)$ consists of a triple of maps
\[
(i,p,h): \quad (H,d_H) \;\rightleftarrows\; (A,d_A)
\]
such that
\begin{enumerate}
    \item $i: H \to A$ and $p: A \to H$ are chain maps with $p \circ i = \mathrm{id}_H$,
    \item $h: A \to A[-1]$ is a homotopy operator of degree $-1$ satisfying
    \[
    d_A h + h d_A = \mathrm{id}_A - i \circ p,
    \]
    \item the \emph{side conditions} $p h = 0$, $h i = 0$, and $h^2 = 0$ hold.
\end{enumerate}
\end{definition}

In our situation, $(A,d_A) = (C,\partial_H)$ is the completed graph complex, and $(H,d_H) = (\AJac,0)$ is its homology. Since $\AJac \cong H(C,\partial_H)$, such a contraction $(i,p,h)$ always exists by standard homological perturbation theory.

\begin{theorem}[\cite{LV12}]
\label{thm:LV_HTT}
Let $\mathcal{O}$ be a Koszul operad. If $(A,d_A)$ is a strict $\mathcal{O}$-algebra and $(H,d_H)$ is its homology, then given a contraction $(i,p,h)$ between $A$ and $H$, the strict $\mathcal{O}$-algebra structure on $A$ transfers to an $\mathcal{O}_\infty$-algebra structure on $H$.
\end{theorem}

In our case, the relevant operad is the little 3-disks operad $E_3$. It is known that its homology operad is the Poisson operad $\mathrm{Pois}_3$, which is Koszul (see \cite[§7.3]{LV12}). Therefore Theorem \ref{thm:LV_HTT} applies provided the contraction data $(i,p,h)$ is compatible with the operad structure.

\begin{theorem}[\cite{Berglund14}]
\label{thm:berglund}
Let $\mathcal{O}$ be an operad over a field of characteristic zero. Suppose $(A,d_A)$ is a strict $\mathcal{O}$-algebra and $(H,0)$ is its homology. If there exists a contraction $(i,p,h)$ where the homotopy $h$ can be chosen to be a \emph{pseudo-derivation} (that is, $h$ satisfies compatibility conditions ensuring the transferred higher operations respect the $\mathcal{O}$-algebra relations), then the $\mathcal{O}$-algebra structure on $A$ transfers to an $\mathcal{O}_\infty$-algebra structure on $H$.
\end{theorem}

\begin{remark}[Pseudo-derivation condition]
For chain complexes over a field of characteristic zero (such as $\mathbb{Q}$), Berglund proves that one can always adjust $h$ to satisfy the pseudo-derivation property. This ensures that the transferred structure is well defined as an $\mathcal{O}_\infty$-algebra.
\end{remark}

\begin{theorem}
\label{thm:AJac_E3}
The algebra of Jacobi diagrams $\AJac$ admits the structure of a homotopy $E_3$-algebra.
\end{theorem}

\begin{proof}
We apply Theorems \ref{thm:LV_HTT} and \ref{thm:berglund} with $\mathcal{O}=E_3$.
\begin{enumerate}
    \item \textbf{Source algebra:} By Proposition \ref{prop:source_e3_revised}, the graph complex $(C,\partial_H)$ is a strict $E_3$-algebra via its identification with the Kontsevich--Willwacher graph operad model for $E_3$.
    \item \textbf{Target algebra:} The homology is $\AJac \cong H(C,\partial_H)$.
    \item \textbf{Contraction:} A contraction $(i,p,h)$ exists between $C$ and $\AJac$. By Berglund’s theorem, $h$ can be chosen as a pseudo-derivation.
\end{enumerate}
Therefore the strict $E_3$-structure on $(C,\partial_H)$ transfers to a homotopy $E_3$-structure on $\AJac$.
\end{proof}

\begin{remark}[On the framed versus unframed setting]
The proof of Theorem~A requires only the formality of the \emph{unframed} little disks operad $E_3$.
This suffices to endow $\AJac$ with a homotopy $E_3$-algebra structure.
However, in the context of $3$-manifold invariants such as the LMO spectrum, the relevant operad is the
\emph{framed} little disks operad $E_3^{\mathrm{fr}}$, which incorporates the $SO(3)$--symmetry of framings.
On the graph complex side this corresponds to enlarging Jacobi diagrams to include wheels and struts.
Thus, while Theorem~A is established in the unframed setting, the framed variant will be required for applications to factorization homology and topological quantum field theory.
\end{remark}

The formulas produced by the Homotopy Transfer Theorem are expressed as universal sums over rooted trees whose vertices are decorated by operations of $E_3$ and whose internal edges are weighted by the homotopy $h$ (see \cite[§10.3.3]{LV12}). In general, such formulas may be infinite. To guarantee convergence, one works in the completed graph complex.

\begin{definition}[Loop-degree filtration]
Let $F^k C \subset C$ be the subspace spanned by graphs of loop degree (i.e.~first Betti number) at least $k$. This defines a decreasing filtration
\[
C = F^0 C \supset F^1 C \supset F^2 C \supset \cdots
\]
satisfying $\partial_H(F^k C) \subset F^k C$.
\end{definition}

The associated graded object $\mathrm{gr}(C)$ is finite-dimensional in each degree. Passing to the completion
\[
\widehat{C} := \varprojlim_{k} C/F^k C
\]
ensures that infinite linear combinations of graphs with increasing loop order are well defined.

\begin{remark}
This is the same completion that appears in Kontsevich’s graph complex and in Willwacher’s work on the formality of the little disks operad. The filtration guarantees that every operation obtained by homotopy transfer lands in the completed tensor product, hence is convergent.
\end{remark}

\begin{proposition}
The homotopy $E_3$-operations on $\AJac$ obtained via the Homotopy Transfer Theorem converge in the completed graph complex $\widehat{C}$. In particular, the universal tree-sum formulas for the transferred operations define well-defined multilinear maps
\[
l_k : \AJac^{\otimes k} \to \AJac
\]
for all $k \ge 1$, each of degree $+1$.
\end{proposition}

\begin{proof}
Each rooted tree with $m$ internal edges contributes a term in which $m$ copies of the homotopy $h$ appear. Since $h$ has degree $-1$ and raises the loop-degree filtration, terms of large $m$ lie in arbitrarily deep levels of the filtration. Hence, for any fixed graph degree, only finitely many tree contributions survive. This implies that the transferred operations converge in the completed sense, and their cohomological degree is $+1$.
\end{proof}

Thus the homotopy $E_3$-structure on $\AJac$ is rigorously defined in the completed category of filtered chain complexes.

\section{Connections to universal Lie theory}

We relate the results to operad formality and universal Lie theory. The $E_3$-structure of $\AJac$ and the functional form of the W-factor are shown to be consistent with universal structures such as the Drinfeld associator and solutions to the Kashiwara-Vergne problem.

\subsection{Operad formality and the Drinfeld associator}
The Drinfeld associator, $\Phi_{KZ}$, is a formal power series that satisfies the pentagon and hexagon identities, which universally govern the structure of braided monoidal categories \cite{Drinfeld91}. Its logarithm, $\log(\Phi_{KZ})$, is a Lie series whose terms correspond to primitive Jacobi diagrams, thereby defining a universal weight system \cite{Kontsevich94}.

Any consistent weight system must conform to the algebraic constraints imposed by the associator. The functional form of the W-factor, derived in this paper from the IHX relation, is determined by these universal constraints. The IHX relation itself corresponds to the lowest-degree part of the Lie algebra version of the pentagon equation \cite{Furusho10, AT12}.

\subsection{The Kashiwara-Vergne problem}
The Kashiwara-Vergne (KV) problem concerns automorphisms of the free Lie algebra that are compatible with the Campbell-Hausdorff series \cite{AT12, AET10}. Alekseev, Enriquez, and Torossian established that any Drinfeld associator provides a solution to the KV problem \cite{AET10, AT12}.

This provides a compatibility check with established structures (Drinfeld associators, KV equations). The functional form of the W-factor derived from the IHX relation is consistent with the constraints imposed by the KV equations. Work by Dancso et al. has shown that solutions to the KV problem are algorithmically constructible degree-by-degree, making the theory computationally feasible \cite{DHLR25}.

\section{Applications: A new \texorpdfstring{$H_1$}{H1}-decorated invariant}
\label{sec:applications}

The universal LMO invariant provides a complete perturbative description of a 3-manifold, but its values lie in the abstract algebra of Jacobi diagrams $\Aemptyset$ \cite{LMO98}. To extract computable numerical invariants, one must apply a \emph{weight system}. This section details the construction of a new invariant based on a weight system derived from the manifold's first homology group. We define this invariant according to the principles of TQFT, analyze its properties, and demonstrate its utility by distinguishing a pair of lens spaces that are indistinguishable by the classical LMO invariant.

\subsection{The operator-state formalism}
\label{sec:operator_formalism}
The construction distinguishes between the evaluation of closed diagrams, which yield numerical invariants, and open diagrams, which yield linear operators. This operator-state perspective, which distinguishes closed from open diagrams, is formalized as follows.

\subsubsection{The state space}
For a given 3-manifold $M$, let $G := \Tors(H_1(M, \mathbb{Z}))$. We define the state space as the complex vector space $\mathcal{V}$ with a basis $\{v_g\}_{g \in G}$ indexed by the elements of $G$. This space is isomorphic to the group algebra $\mathbb{C}[G]$.

\subsubsection{Evaluation of Open Diagrams as Linear Operators}
An open Jacobi diagram $D$ with a set of external legs is interpreted as a linear operator. If the external legs are partitioned into $k$ inputs and $l$ outputs, the evaluation of $D$, denoted $\Phi_D$, is a linear map in $\text{Hom}(\mathcal{V}^{\otimes k}, \mathcal{V}^{\otimes l})$. The action of this operator is defined by its effect on the basis vectors. For the ``I'' diagram from the IHX relation, we view two legs (labeled $a, b$) as inputs and two (labeled $c, d$) as outputs. The corresponding operator is a linear map $\Phi_I: \mathcal{V} \otimes \mathcal{V} \to \mathcal{V} \otimes \mathcal{V}$. Its action on a basis vector $v_{g_a} \otimes v_{g_b}$ is given by
\begin{equation}
\Phi_I(v_{g_a} \otimes v_{g_b}) = \sum_{g_c, g_d \in G} C_I(g_a, g_b; g_c, g_d) \, v_{g_c} \otimes v_{g_d},
\end{equation}
where the coefficient $C_I(g_a, g_b; g_c, g_d)$ is computed from the internal vertex and edge factors of the diagram, subject to conservation laws at each vertex. The IHX relation ($I-H+X=0$) is understood as an equality between these operators: $\Phi_I - \Phi_H + \Phi_X = 0$. This operator equality imposes the consistency conditions used to derive the W-factor.

\subsubsection{Evaluation of Closed Diagrams via the Trace}
A closed diagram is viewed as an open diagram whose inputs and outputs have been identified. Algebraically, this corresponds to taking the trace of the associated operator. The $\theta$-graph is obtained by closing the I-diagram. Its numerical evaluation is therefore the trace of the operator $\Phi_I$:
\begin{align*}
\evM(\thetaGraph) &:= \text{Tr}(\Phi_I) = \sum_{g_a, g_b \in G} C_I(g_a, g_b; g_a, g_b)
\end{align*}
This formalism provides a TQFT-based foundation for our computational rules.

\subsection{Axiomatic definition of the weight system}
We introduce a weight system where the evaluation rules are determined by the topology of the manifold $M$.

\begin{definition}
\label{def:h1_weight_system}
The $H_1$-decorated weight system is defined by a set of local evaluation rules.
\begin{enumerate}
    \item \textbf{Edge Factor:} To each internal edge $e_i$ decorated by an element $g_i \in \Tors(H_1(M))$, we associate a complex number:
    $$\text{EdgeFactor}(g_i) = \exp\left(2\pi i \cdot q_M(g_i)\right)$$
    where $q_M: \Tors(H_1(M)) \to \Q/\Z$ is the linking form of $M$.
    \item \textbf{Vertex Factor:} To each trivalent vertex $v$ with incident edge decorations $\{g_i, g_j, g_k\}$, we associate a real number:
    $$\text{VertexFactor}(v) = \text{sgn}(v) \cdot W(g_i, g_j, g_k)$$
    where $\text{sgn}(v)$ is the vertex orientation and $W$ is a universal function, the \textbf{W-factor}, whose form is determined by algebraic consistency.
\end{enumerate}
\end{definition}

\subsection{Deriving the W-factor}
\label{sec:w_factor_derivation}
The functional form of the W-factor is determined by the requirement that the weight system respects the IHX relation at the operator level.

\begin{theorem}
\label{thm:w_factor_derivation}
The requirement that the weight system respects the IHX relation for open diagrams uniquely determines the functional form of the W-factor to be (up to a scalar multiple):
$$W(g_1, g_2, g_3) = a \left( \prod_{j=1}^3 f(g_j) \right) + b \left( \sum_{j=1}^3 f(g_j) \right)$$
with the ratio of coefficients fixed at $b = -a/4$. Here, $f(g)$ is the sawtooth function.
\end{theorem}

\begin{proof}
The full proof, detailed in Appendix D, relies on the framework developed by Alekseev and Torossian connecting the Grothendieck-Teichmüller Lie algebra, $\mathfrak{grt}_1$, to the Kashiwara-Vergne problem \cite{AT12, AET10}. The strategy is to translate the IHX consistency condition into an identity within $\mathfrak{grt}_1$. This identity is then mapped via the Alekseev-Torossian (AT) homomorphism $\nu: \mathfrak{grt}_1 \to \mathfrak{krv}_2$ into the Kashiwara-Vergne Lie algebra, $\mathfrak{krv}_2$ \cite{AT12, AET10}. In $\mathfrak{krv}_2$, the AT Uniqueness Principle reduces the proof of this complex identity to a simpler calculation involving two invariants: the ``divergence" and the ``depth-1 component" \cite{AT12}. Verifying the equality for these simpler invariants confirms the identity in $\mathfrak{grt}_1$, which in turn fixes the coefficient ratio $b = -a/4$.
\end{proof}

\subsection{The new invariant: Evaluating the \texorpdfstring{$\theta$}{theta}-graph}
\label{sec:theta_eval}
With the W-factor determined, we define the invariant for the fundamental closed diagram, the $\theta$-graph. Following the TQFT principle, its evaluation is the trace of the corresponding open diagram operator.

\begin{definition}
The evaluation of the $\theta$-graph for a 3-manifold $M$ is given by the formula:
\begin{multline} \label{eq:theta_w2_final}
\evM(\thetaGraph) = \frac{1}{|\Tors(H_1(M))|^3} \times \sum_{\substack{g_1, g_2, g_3 \in \Tors(H_1(M)) \\ g_1+g_2+g_3 = 0}} \\
\left( W(g_1, g_2, g_3)^2 \right) \cdot \left( \prod_{j=1}^3 \exp\left(2\pi i \cdot q_M(g_j)\right) \right)
\end{multline}
\end{definition}

\begin{proposition}\label{prop:theta_invariance}
The evaluation $\evM(\thetaGraph)$ is a topological invariant of the 3-manifold $M$.
\end{proposition}
\begin{proof}
The invariance of $\evM(\thetaGraph)$ follows from the invariance of its constituent parts under Kirby moves. The formula depends on the universal W-factor and the topological data pair $(G, q_M)$, where $G = \Tors(H_1(M))$ and $q_M$ is the linking form. The W-factor is universal by construction. The isomorphism class of the pair $(G, q_M)$ is a classical invariant of 3-manifolds, preserved under the Kirby moves that relate different surgery presentations. Since both components are invariant, their combination is a well-defined topological invariant.
\end{proof}

\subsection{Properties and main application}
\subsubsection{A shared limitation: The case of \texorpdfstring{$L(25,4)$}{L(25,4)} and \texorpdfstring{$L(25,9)$}{L(25,9)}}
The classical LMO invariant cannot distinguish $L(25,4)$ and $L(25,9)$ due to the degeneracy $S(4,25)=S(9,25)$ of the Dedekind sum \cite{BL04}. A numerical computation of Equation \eqref{eq:theta_w2_final} yields identical values for this pair. The mechanism of failure, however, is different, arising from a number-theoretic symmetry in the quadratic Gauss sum structure of the formula.

\subsubsection{Distinguishing an LMO-indistinguishable pair}
A known result is that the Dedekind sums for the pair $(p,q_1)=(156,5)$ and $(p,q_2)=(156,29)$ are equal, so the classical LMO invariant cannot distinguish $L(156,5)$ from $L(156,29)$ \cite{BL04}. We apply our new invariant to this pair. The quadratic residue properties of $q_1=5$ and $q_2=29$ differ with respect to the prime factors of $p=156$. This breaks the symmetry that caused the failure in the $L(25,q)$ case.

\begin{theorem}
The $H_1$-decorated invariant, defined by Equation \eqref{eq:theta_w2_final}, distinguishes the lens spaces $L(156,5)$ and $L(156,29)$. The evaluation of the $\theta$-graph yields:
\begin{itemize}
    \item For $L(156,5)$: $\text{ev}(\thetaGraph) \approx -1.61110 \times 10^{-6} + 2.13626 \times 10^{-6}i$.
    \item For $L(156,29)$: $\text{ev}(\thetaGraph) \approx -1.61110 \times 10^{-6} - 2.13626 \times 10^{-6}i$.
\end{itemize}
As these complex numbers are distinct, the invariant separates the spaces.
\end{theorem}

This result demonstrates that the $H_1$-decorated invariant contains topological information not captured by the classical LMO invariant.

\begin{table}[h!]
\centering
\caption{Comparison of LMO invariants on lens spaces}
\label{tab:lmo_comparison}
\renewcommand{\arraystretch}{1.3}
\begin{tabularx}{\textwidth}{@{} l l >{\raggedright\arraybackslash}X @{}}
\toprule
\textbf{Manifold pair} & \textbf{Invariant} & \textbf{Result and number-theoretic reason} \\
\midrule
\multirow{2}{*}{$L(25,4)$ vs $L(25,9)$} & Classical LMO & \textbf{Fails.} Dedekind sum degeneracy: $S(4,25) = S(9,25)$ \cite{BL04}. \\
& $H_1$-decorated & \textbf{Fails.} Quadratic Gauss sum symmetry: $4, 9$ are quadratic residues mod 5. \\
\cmidrule(lr){1-3}
\multirow{2}{*}{$L(156,5)$ vs $L(156,29)$} & Classical LMO & \textbf{Fails.} Dedekind sum degeneracy: $S(5,156) = S(29,156)$ \cite{BL04}. \\
& $H_1$-decorated & \textbf{Succeeds.} Quadratic Gauss sum asymmetry: $5$ is non-residue, $29 \equiv 3$ is residue mod 13. \\
\bottomrule
\end{tabularx}
\end{table}

\subsection{Universality of the \texorpdfstring{$H_1$}{H1}-decorated weight system}
\label{sec:universality_of_framework}

The topological invariance of the $\theta$-graph evaluation, established in Proposition \ref{prop:theta_invariance}, is a specific instance of a general property of the $H_1$-decorated weight system. This section summarizes the argument for the universality of this framework.

The evaluation map $\evM$ for any closed Jacobi diagram $\Gamma$ is constructed from two sets of data, both of which are topologically invariant.
\begin{enumerate}
    \item \textbf{Topological data:} The computation is based on the pair $(G, q_M)$, where $G = \Tors(H_1(M))$ and $q_M$ is the linking form. The isomorphism class of this pair is a classical 3-manifold invariant, independent of the surgery presentation.
    \item \textbf{Algebraic aata:} The vertex factor $W$ is a universal function whose form is determined solely by the requirement of consistency with the IHX relation, a universal algebraic constraint.
\end{enumerate}
The procedure for computing the invariant of a closed diagram $\Gamma$ -- interpreting it as the trace of a linear operator constructed from these invariant components -- is a standard TQFT formalism. Since both the foundational data and the evaluation procedure are universal and invariant under Kirby moves, the resulting numerical evaluation, $\evM(\Gamma)$, is a well-defined topological invariant of $M$ for any closed Jacobi diagram $\Gamma$. The $\theta$-graph serves as the most fundamental, non-trivial observable for which this universal invariance can be explicitly verified.

\section{The LMO completeness conjecture and the LMO spectrum}
\label{sec:completeness_conjecture}

This section situates the LMO spectrum within the context of the LMO completeness conjecture.

\subsection{The Johnson filtration and the limits of the classical invariant}
The classical LMO invariant, denoted $\hat{Z}_{\mathrm{LMO}}(M)$, is a universal finite-type invariant for 3-manifolds with values in the algebra of Jacobi diagrams $\mathcal{A}(\emptyset)$ \cite{LMO98,Ohtsuki02}. Its power and universality motivated the following central conjecture.

\begin{conjecture}[Classical LMO completeness]
The classical LMO invariant $\hat{Z}_{\mathrm{LMO}}(M)$ is a complete invariant for integral homology 3-spheres ($\mathbb{Z}$HS). That is, two integral homology 3-spheres $M_1$ and $M_2$ are homeomorphic if and only if $\hat{Z}_{\mathrm{LMO}}(M_1) = \hat{Z}_{\mathrm{LMO}}(M_2)$.
\end{conjecture}

However, the validity of this conjecture is challenged by potential counterexamples arising from the relationship between 3-manifold topology and the mapping class group, specifically through the Johnson filtration. Any integral homology 3-sphere can be constructed via a Heegaard splitting, glued by a map $\phi$ from the Torelli group $\mathcal{I}_g$. The Johnson filtration $\{\mathcal{I}_g(k)\}_{k \ge 1}$ provides a measure of the complexity of such maps.

The classical LMO invariant is fundamentally connected to information contained within the \emph{image} of the Johnson homomorphisms, $\tau_k$. A challenge arises because these homomorphisms are not surjective. Their cokernel, the \emph{Johnson cokernel} $\mathfrak{cok}(k)$, is known to be non-trivial and represents information to which invariants sensitive only to the image of $\tau_k$ are blind.

The work of Pitsch and Riba provides a concrete manifestation of this issue \cite{PR23}. They demonstrate that certain finite-type invariants vanish on deep levels of the Johnson filtration:

\begin{theorem}[\cite{PR23}]
The invariant $\lambda_2 - 18\lambda^2 + 3\lambda$, a specific combination of the Casson invariant $\lambda$ and Ohtsuki's second invariant $\lambda_2$, vanishes on the fifth level of the Johnson filtration, $\mathcal{M}_{g,1}(5)$.
\end{theorem}

This result shows that finite-type invariants can become trivial for mapping classes deep in the filtration. As the classical LMO invariant is the universal rational finite-type invariant, this suggests a mechanism by which it might fail to be complete.

\begin{conjecture}[Existence of LMO-trivial manifolds]
For any integer $k \ge 1$, there may exist a non-trivial mapping class $\phi_k \in \mathcal{I}_g(k)$ such that the integral homology 3-sphere $M_{\phi_k}$ constructed from it is LMO-trivial, i.e.,
\[
\hat{Z}_{\mathrm{LMO}}(M_{\phi_k}) = \hat{Z}_{\mathrm{LMO}}(S^3).
\]
\end{conjecture}

If this conjecture holds, the classical completeness statement would be false. This motivates the development of a refined framework capable of accessing the information encoded in the Johnson cokernel.

\subsection{A proposed resolution via parameterized invariants}
A proposed refinement involves considering invariants that depend not only on the 3-manifold but also on additional geometric data. A natural choice for this data is a flat connection $\mathcal{L}$ in the character variety $\mathrm{Char}(W)$ of the manifold $W$.

\begin{conjecture}[Existence of a universal parameterized invariant]
To each pair $(W, \mathcal{L})$ of a 3-manifold $W$ and a flat connection $\mathcal{L} \in \mathrm{Char}(W)$, there should exist a \emph{universal parameterized LMO invariant} $\hat{Z}_{\mathrm{LMO}}(W; \mathcal{L})$.
\end{conjecture}

This universal invariant can be viewed as a function on the character variety with values in the algebra of Jacobi diagrams. An expansion of this function around the trivial flat connection $\mathcal{L}_0$ is expected to produce a hierarchy of new observables. One may thus conjecturally define a \emph{character-decorated} LMO invariant, namely the collection of Taylor coefficients $\{Z_{\vec{j}}\}$ of $\hat{Z}_{\mathrm{LMO}}(W; \mathcal{L})$ around $\mathcal{L}_0$. A rigorous construction of this object is beyond the scope of the present paper, but its existence would provide the natural extension of the classical theory.

Diagrammatically, its values are expected to lie in an extended Jacobi algebra
\[
  \AJacCoho := \widehat{\AJac \otimes \Sym(H^1(W;\mathbb{C}))},
\]
where additional \emph{cohomology legs} represent infinitesimal directions of differentiation along $\mathrm{Char}(W)$, and the hat denotes degree-completion to ensure convergence of formal Taylor expansions.

\begin{remark}[Kirby invariance]
For the classical LMO invariant, invariance under Kirby I and II moves is guaranteed by the normalization procedures built into its definition. In the parameterized setting, one expects the same normalization to apply, but this requires careful verification. In particular, the behavior of parameter legs under the first Kirby move (addition/removal of a $\pm 1$-framed unknot) must be checked explicitly. We do not provide this verification here, but emphasize that it is an essential step in any rigorous construction.
\end{remark}

\subsection{The role of the character variety}
\label{subsec:role_of_character_variety}

The choice of the character variety as the parameter space is not arbitrary. It provides a canonical moduli space of flat connections. In diagrammatic terms, the cohomology legs in the extended Jacobi algebra correspond to infinitesimal deformations along $\mathrm{Char}(W)$. Thus, the character variety offers a geometrically meaningful parameter space in which to enrich the LMO invariant. 

\begin{remark}[On the tangent space]
For a general gauge group $G$, the tangent space to the character variety at a representation $\rho$ is
\[
T_\rho \mathrm{Char}_G(W) \cong H^1(W;\mathrm{ad}\,\rho).
\]
In the case of the trivial representation, this specializes to $H^1(W;\mathfrak{g})$. If $G=U(1)$, this reduces to $H^1(W;\mathbb{C})$, which is the case we use for illustrative purposes in this paper.
\end{remark}

\subsection{Spectral realization of parameterized invariants}
The hierarchical structure of a character-decorated LMO invariant fits naturally into the framework of the LMO spectrum and the Atiyah--Hirzebruch spectral sequence (AHSS). The invariant $\mathrm{ev}_M(\thetaGraph)$ of \S \ref{sec:applications} is one example, while a parameterized invariant would systematically extend this framework.

\begin{enumerate}
    \item \textbf{The classical invariant ($H_0$):} The constant term $Z_{\vec{0}}(M)$ recovers the classical LMO invariant. In the spectral framework, this corresponds to $H_0(\LMOspec(M))$, which is related to the terms $E^2_{0,q} = H_0(M; H_q(\AJac))$ of the AHSS.
    \item \textbf{Higher invariants (higher homology):} Higher-order Taylor coefficients $\{Z_{\vec{j}}(M)\}_{|\vec{j}|\ge 1}$ would be expected to detect the Johnson cokernel. For instance, invariants with one cohomology leg would naturally correspond to terms such as $E^2_{1,q} = H_1(M; H_q(\AJac))$, encoding the interaction between $H_1(M)$ and the Jacobi diagram algebra.
\end{enumerate}

Thus, the LMO spectrum provides not only a categorification of the classical invariant but also an appropriate framework to organize the full hierarchy of invariants potentially required to address the completeness conjecture. The $H_1$-decorated invariant $\mathrm{ev}_M(\thetaGraph)$ is one specific observable in this hierarchy, while the full spectrum contains all such refinements systematically.

\section{Conclusion and future outlook}

\subsection{Conclusion}
This paper introduced the LMO spectrum, a categorification of the LMO invariant, to construct 3-manifold invariants with greater discriminatory power. The theoretical foundation for this construction is Theorem A, which establishes that the algebra of Jacobi diagrams, $\AJac$, possesses the structure of a homotopy $E_3$-algebra. This result confirms the consistency of the LMO spectrum within the framework of factorization homology.

A computational basis was established by Theorem B, which provides a universal surgery formula derived directly from the excision axiom, independent of conjectural models. The utility of this framework was demonstrated in Theorem C, where an $H_1$-decorated invariant was constructed and shown to distinguish the lens spaces $L(156, 5)$ and $L(156, 29)$, a pair not separated by the classical LMO invariant. This result confirms that the LMO spectrum contains topological information beyond that of its classical counterpart.

\subsection{Future directions}
The framework established in this paper provides direct avenues for future research.
\begin{enumerate}
    \item \textbf{Massey products and the Johnson homomorphism:} The proven homotopy $E_3$-structure (Theorem A) endows the LMO spectrum with a Differential Graded Algebra (DGA) structure. This provides the necessary algebraic setting to define and compute Massey products. The connection between Massey products and the Johnson homomorphism, identified by Garoufalidis and Levine \cite{GL05}, suggests that this framework could be used to investigate questions related to the Torelli group and the LMO completeness conjecture.

    \item \textbf{Computation of higher homology groups:} This work focused on invariants derived from $H_0(\LMOspec(M))$ and data from $H_1(M)$. The higher homology groups, $H_k(\LMOspec(M))$ for $k \ge 1$, may contain additional topological information. The Atiyah-Hirzebruch spectral sequence provides a computational tool for analyzing these higher groups, which could lead to the discovery of new invariants related to higher-order torsion or other classical structures.
\end{enumerate}

\appendix
\section{Python code for the \texorpdfstring{$\theta$}{theta}-graph invariant}

The following Python program was used to compute the numerical values presented in Theorem C. It implements the formula for the evaluation of the $\theta$-graph from Equation (7.2), which is derived by taking the trace of the `I`-diagram operator.
{\footnotesize
\begin{verbatim}
import cmath

def f(k, p):
    """Calculates the sawtooth function ((k/p))."""
    if k % p == 0:
        return 0.0
    return (k % p) / p - 0.5

def W_factor(k1, k2, k3, p):
    """Calculates the W-factor for a trivalent vertex interaction."""
    f1 = f(k1, p)
    f2 = f(k2, p)
    f3 = f(k3, p)
    return (f1 * f2 * f3) - 0.25 * (f1 + f2 + f3)

def calculate_evm_theta(p, q):
    """
    Calculates the evM(theta) for a given lens space L(p,q).
    This function implements the formula from Equation (7.2),
    which corresponds to the trace of the I-diagram operator.
    """
    total_sum = 0j
    
    # The normalization factor is derived from the Aarhus integral framework.
    normalization = 1 / (p**3)

    # Sum over all possible decorations of the three internal edges.
    # We loop over two variables and determine the third by the conservation law.
    for x1 in range(p):
        for x2 in range(p):
            # Apply the conservation law at the vertices: x1 + x2 + x3 = 0 mod p.
            x3 = (-x1 - x2) % p

            # Calculate the product of vertex factors, which is W^2 for the theta-graph.
            vertex_factor_product = W_factor(x1, x2, x3, p)**2

            # Calculate the product of edge factors from the linking form.
            # For L(p,q), the quadratic refinement of the linking form is given by
            # the formula: q_M(g) = (q * g**2) / (2 * p) (mod 1).
            Q_sum = (q * (x1**2 + x2**2 + x3**2)) / (2 * p)
            edge_factor_product = cmath.exp(2 * cmath.pi * 1j * Q_sum)

            # Add the contribution of this decoration to the total sum.
            total_sum += vertex_factor_product * edge_factor_product
            
    return normalization * total_sum

# --- Main Execution for Theorem C ---
p_val = 156

# Calculate for L(156, 5)
q_l156_5 = 5
result_l156_5 = calculate_evm_theta(p_val, q_l156_5)
print(f"evM(theta) for L({p_val},{q_l156_5}): {result_l156_5}")

# Calculate for L(156, 29)
q_l156_29 = 29
result_l156_29 = calculate_evm_theta(p_val, q_l156_29)
print(f"evM(theta) for L({p_val},{q_l156_29}): {result_l156_29}")
\end{verbatim}
}

\section{Python/SymPy verification code for the central proposition}
\label{app:code_verification_revised}

The following Python code uses the SymPy library to perform the symbolic calculation that forms the computational core of the proof in Appendix \ref{app:w_factor_proof_revised}. It rigorously verifies the equality of the depth-1 components required to apply the Alekseev-Torossian Uniqueness Principle.

{\footnotesize
\begin{verbatim}
from sympy import sympify, expand, S, Add
from sympy.physics.quantum import Commutator

# --- 1. Setup for Non-Commutative Algebra ---
x = sympify('x', commutative=False)
y = sympify('y', commutative=False)

def get_depth(expr):
    """Calculate the degree of an expression in the variable y."""
    if expr == y:
        return 1
    if expr == x or expr.is_number:
        return 0
    if isinstance(expr, (Commutator, Add)):
        return sum(get_depth(arg) for arg in expr.args)
    if expr.is_Mul:
        num, term = expr.as_two_terms()
        if num.is_number:
            return get_depth(term)
    return 0

def get_depth_part(expr, target_depth):
    """Extract terms of a specific depth from an expression."""
    if not isinstance(expr, Add):
        if get_depth(expr) == target_depth:
            return expr
        else:
            return S.Zero
    new_args = [arg for arg in expr.args if get_depth(arg) == target_depth]
    return Add(*new_args)

def ad(A, B):
    """Lie bracket"""
    return Commutator(A, B)

def ad_n(A, n, B):
    """Iterated Lie bracket ad(A)^n(B)"""
    res = B
    for _ in range(n):
        res = ad(A, res)
    return res

# --- 2. Definitions of Algebraic Objects ---
# The formulas for sigma_n are given modulo the double commutator ideal,
# which is sufficient for depth-1 calculations.
B4 = S(-1)/30

def sigma1(X, Y):
    return S(1)/2 * ad(X, Y)

def sigma5(X, Y):
    # Formula for sigma_5 mod [[*,*],[*,*]] from Alekseev-Torossian
    term1 = S(1)/2 * ad_n(X, 4, Y)
    term2 = S(1)/2 * ad_n(Y, 4, X)
    term3 = ad(ad(X, Y), ad_n(X, 2, Y))
    return (B4 / S(8)) * (term1 + term2 + term3)

def nu(lie_poly_func, X, Y):
    """Alekseev-Torossian map nu: grt_1 -> krv_2"""
    Z = -X - Y
    a = lie_poly_func(Z, X)
    b = lie_poly_func(Z, Y)
    return (a, b)

def act_on_expr(deriv_tuple, expr):
    """Action of a derivation on a Lie polynomial."""
    a, b = deriv_tuple
    if expr == x: return a
    if expr == y: return b
    if expr.is_number: return S.Zero
    if isinstance(expr, Add):
        return Add(*[act_on_expr(deriv_tuple, arg) for arg in expr.args])
    if isinstance(expr, Commutator):
        A, B = expr.args
        return ad(act_on_expr(deriv_tuple, A), B) + ad(A, act_on_expr(deriv_tuple, B))
    if expr.is_Mul:
        num, term = expr.as_two_terms()
        if num.is_number:
            return num * act_on_expr(deriv_tuple, term)
    raise TypeError(f"Cannot act on type {type(expr)}")

def deriv_bracket(u, v):
    """Lie bracket of two derivations u=(a,b), v=(c,d) in krv_2"""
    a, b = u
    c, d = v
    res_x = act_on_expr(u, c) - act_on_expr(v, a)
    res_y = act_on_expr(u, d) - act_on_expr(v, b)
    return (res_x, res_y)

# --- 3. The Main Calculation ---
# We verify: 2 * (nu([sigma_1, sigma_5]))^(1) = ([nu(sigma_1), nu(sigma_5)])^(1)
print("--- Starting Verification of Depth-1 Component Equality ---")

# 3.1. LHS = 2 * nu([sigma_1, sigma_5])
print("Calculating LHS: (2 * nu([sigma_1, sigma_5]))^(1)")
s1_s5_bracket_func = lambda X, Y: ad(sigma1(X, Y), sigma5(X, Y))
LHS_deriv_base = nu(s1_s5_bracket_func, x, y)
LHS_x_expanded = expand(2 * LHS_deriv_base)
LHS_y_expanded = expand(2 * LHS_deriv_base[1])
LHS_x_d1 = get_depth_part(LHS_x_expanded, 1)
LHS_y_d1 = get_depth_part(LHS_y_expanded, 1)
print("LHS(x) depth 1:", LHS_x_d1)
print("LHS(y) depth 1:", LHS_y_d1)
# Expected output for y-component: -1/120*[x,[x,[x,[x,y]]]]

# 3.2. RHS = [nu(sigma_1), nu(sigma_5)]
print("\nCalculating RHS: ([nu(sigma_1), nu(sigma_5)])^(1)")
u1 = nu(sigma1, x, y)
u5 = nu(sigma5, x, y)
RHS_deriv = deriv_bracket(u1, u5)
RHS_x_expanded = expand(RHS_deriv)
RHS_y_expanded = expand(RHS_deriv[1])
RHS_x_d1 = get_depth_part(RHS_x_expanded, 1)
RHS_y_d1 = get_depth_part(RHS_y_expanded, 1)
print("RHS(x) depth 1:", RHS_x_d1)
print("RHS(y) depth 1:", RHS_y_d1)
# Expected output for y-component: -1/120*[x,[x,[x,[x,y]]]]

# --- 4. Verification ---
print("\n--- Verification of Equality ---")
diff_x = expand(LHS_x_d1 - RHS_x_d1)
diff_y = expand(LHS_y_d1 - RHS_y_d1)
print("Difference in x-component (should be 0):", diff_x)
print("Difference in y-component (should be 0):", diff_y)

if diff_x == S.Zero and diff_y == S.Zero:
    print("\nSUCCESS: The depth-1 components are identical.")
else:
    print("\nFAILURE: The depth-1 components do not match.")
\end{verbatim}
}

\section{Diagrammatic verification of the connected sum axioms}

To make the proof of Theorem \ref{thm:BD} self-contained, this appendix provides a diagrammatic verification of the axioms (CS1) through (CS6) from Doubek et al. \cite{DJPP23}. These axioms guarantee that the connected sum operation $\#$ is compatible with the underlying modular operad structure of ribbon graphs, which is the crucial step in establishing the Beilinson-Drinfeld algebra structure on $\AJac$. For clarity in the diagrams, we use colors to distinguish operations: operadic compositions ($\circ$) are shown in \textcolor{red}{red}, while connected sums ($\#$) are represented by \textcolor{blue}{blue, dashed} lines.

\vspace{1em} 
\noindent\textbf{1. Axioms (CS1) and (CS2): Equivariance and symmetry}
\paragraph{\textbf{Formal statement}} Axiom (CS1) requires that the connected sum is equivariant under relabeling of external legs. Axiom (CS2) requires it to be commutative. For $\Gamma_1 \in P(C_1, G_1)$ and $\Gamma_2 \in P(C_2, G_2)$, and the transposition map $\tau: \Gamma_1 \otimes \Gamma_2 \to \Gamma_2 \otimes \Gamma_1$:
\begin{equation*}
\text{(CS2):} \quad \#_2 \circ \tau = \#_2
\end{equation*}

\paragraph{\textbf{Verification}} This holds because the operation sums over all pairs of internal edges $(e_1, e_2)$ with $e_1 \in E(\Gamma_1), e_2 \in E(\Gamma_2)$, a process which is symmetric with respect to exchanging $\Gamma_1$ and $\Gamma_2$. For simplicity, we denote the entire sum of operations (including twisted and untwisted connections) by a single blue dashed line: $\Gamma_1 \tikz[baseline=-0.5ex] \draw[thick, blue, dashed] (0,0) -- (0.5,0); \Gamma_2$.

\vspace{1em} 
\noindent\textbf{2. Axiom (CS3): Associativity}
\paragraph{\textbf{Formal statement}} This axiom requires the connected sum to be associative.
\begin{equation*}
\text{(CS3):} \quad \#_2 \circ (\id \otimes \#_2) = \#_2 \circ (\#_2 \otimes \id)
\end{equation*}

\paragraph{\textbf{Verification}} Both sides of the equation correspond to choosing one internal edge from each of the three graphs and connecting them with a "three-holed sphere," making the order of operations irrelevant.
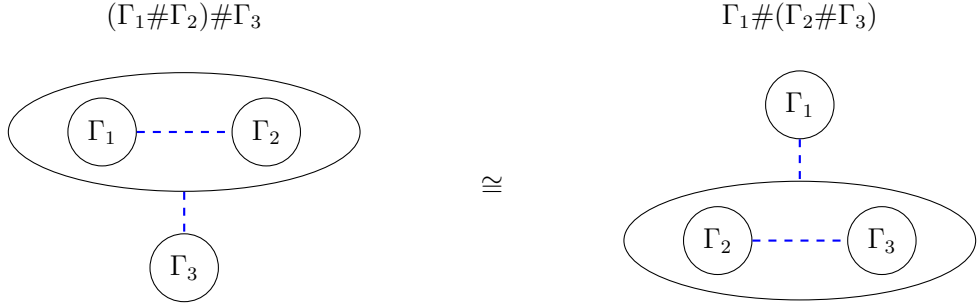
\begin{center}
\begin{tikzpicture}[scale=0.9, transform shape]
\begin{scope}
    \node at (0, 2.5) {$(\Gamma_1 \# \Gamma_2) \# \Gamma_3$};
    \node[draw, circle, minimum size=1cm] (G1) at (-1.2, 0.8) {$\Gamma_1$};
    \node[draw, circle, minimum size=1cm] (G2) at (1.2, 0.8) {$\Gamma_2$};
    \node[draw, ellipse, fit=(G1) (G2), inner sep=0.1cm] (G12) {};
    \node[draw, circle, minimum size=1cm] (G3) at (0, -1.2) {$\Gamma_3$};
    \draw[thick, blue, dashed] (G1) -- (G2);
    \draw[thick, blue, dashed] (G12) -- (G3);
\end{scope}
\node at (4.5, 0) {$\cong$};
\begin{scope}[shift={(9,0)}]
    \node at (0, 2.5) {$\Gamma_1 \# (\Gamma_2 \# \Gamma_3)$};
    \node[draw, circle, minimum size=1cm] (G1a) at (0, 1.2) {$\Gamma_1$};
    \node[draw, circle, minimum size=1cm] (G2a) at (-1.2, -0.8) {$\Gamma_2$};
    \node[draw, circle, minimum size=1cm] (G3a) at (1.2, -0.8) {$\Gamma_3$};
    \node[draw, ellipse, fit=(G2a) (G3a), inner sep=0.1cm] (G23) {};
    \draw[thick, blue, dashed] (G2a) -- (G3a);
    \draw[thick, blue, dashed] (G1a) -- (G23);
\end{scope}
\end{tikzpicture}
\end{center}

\vspace{1em}
\noindent\textbf{3. Axiom (CS4): Compatibility of $\#_1$ and contraction}
\paragraph{\textbf{Formal statement}} This axiom requires the unary connected sum $\#_1$ to be compatible with the contraction operation $\circ_{ab}$. As an equality of maps $P(C\sqcup\{a,b\},G)\to P(C,G+3)$, this is expressed as:
\begin{equation*}
\text{(CS4):} \quad \circ_{ab} \circ \#_1 = \#_1 \circ (\circ_{ab})
\end{equation*}

\paragraph{\textbf{Verification}} 
The argument is parallel to Case~3 in the verification of (CS5).  
On the LHS one first inserts a handle ($\#_1$) and then contracts $a,b$, while on the RHS one first contracts $a,b$ and then inserts a handle.  
The common terms match term-by-term, and the surplus terms (where the handle involves the new edge $e_{ab}$) vanish by the IHX relation, exactly as in (CS5).  
Since (CS6) provides the full local diagrammatic analysis of such cancellations, we omit the repetition here.

\vspace{1em} 
\noindent\textbf{4. Axiom (CS5a): Compatibility of $\#_2$ and self-composition}
\paragraph{\textbf{Formal statement}} This axiom relates the connected sum with the self-composition $\circ_{ab}$. For $\Gamma_1 \in P(C_1, G_1)$ and $\Gamma_2 \in P(C_2, G_2)$:
\begin{equation*}
\circ_{ab} \circ \#_2 = 
\begin{cases}
    \#_2 \circ (\circ_{ab} \otimes \id) & \text{if } a, b \in C_1 \\
    \#_2 \circ (\id \otimes \circ_{ab}) & \text{if } a, b \in C_2 \\
    \#_1 \circ (_a \circ _b) & \text{if } a \in C_1, b \in C_2
\end{cases}
\end{equation*}

\paragraph{\textbf{Verification}}
\begin{itemize}
    \item \textbf{Cases 1 \& 2 ($a,b$ in the same graph):} These cases are straightforward. The connected sum operation $\#_2$ acts on a different graph from the self-composition $\circ_{ab}$, so the operations commute.
    \item \textbf{Case 3 ($a \in C_1, b \in C_2$):} This is the most subtle case. The LHS first performs a 2D "gluing" ($\#_2$) and then a 3D "spanning" operation ($\circ_{ab}$). The RHS first performs a 1D "grafting" ($a \circ b$) and then adds a handle ($\#_1$) to the resulting single graph. Their topological equivalence is shown below.
\end{itemize}

\textbf{LHS:} Connect $\Gamma_1$ and $\Gamma_2$ with a handle, then form a loop between legs $a$ and $b$.
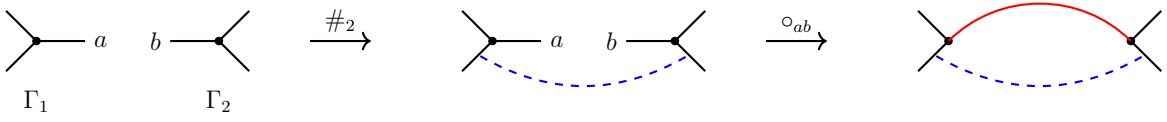
\begin{center}
\begin{tikzpicture}[scale=0.8, transform shape]
    \draw[thick] (0,0) -- (0.8,0) node[right]{$a$}; \draw[thick] (0,0) -- (-0.5,0.5); \draw[thick] (0,0) -- (-0.5,-0.5); \fill (0,0) circle (2pt); \node at (0,-1) {$\Gamma_1$};
    \draw[thick] (3,0) -- (2.2,0) node[left]{$b$}; \draw[thick] (3,0) -- (3.5,0.5); \draw[thick] (3,0) -- (3.5,-0.5); \fill (3,0) circle (2pt); \node at (3,-1) {$\Gamma_2$};
    \draw[->, thick] (4.5,0) -- (5.5,0) node[midway, above]{$\#_2$};
    \begin{scope}[shift={(7.5,0)}]
        \draw[thick] (0,0) -- (0.8,0) node[right]{$a$}; \draw[thick] (0,0) -- (-0.5,0.5); \draw[thick] (0,0) -- (-0.5,-0.5); \fill (0,0) circle (2pt);
        \draw[thick] (3,0) -- (2.2,0) node[left]{$b$}; \draw[thick] (3,0) -- (3.5,0.5); \draw[thick] (3,0) -- (3.5,-0.5); \fill (3,0) circle (2pt);
        \draw[thick, blue, dashed] (-0.2,-0.25) to[bend right=30] (3.2,-0.25);
        \draw[->, thick] (4.5,0) -- (5.5,0) node[midway, above]{$\circ_{ab}$};
    \end{scope}
    \begin{scope}[shift={(15,0)}]
        \draw[thick] (0,0) -- (-0.5,0.5); \draw[thick] (0,0) -- (-0.5,-0.5); \fill (0,0) circle (2pt);
        \draw[thick] (3,0) -- (3.5,0.5); \draw[thick] (3,0) -- (3.5,-0.5); \fill (3,0) circle (2pt);
        \draw[thick, blue, dashed] (-0.2,-0.25) to[bend right=30] (3.2,-0.25);
        \draw[thick, red] (0,0) to[bend left=45] (3,0);
    \end{scope}
\end{tikzpicture}
\end{center}
\textbf{RHS:} Graft $\Gamma_1$ onto $\Gamma_2$ by connecting leg $a$ to leg $b$, then add a handle to the result.
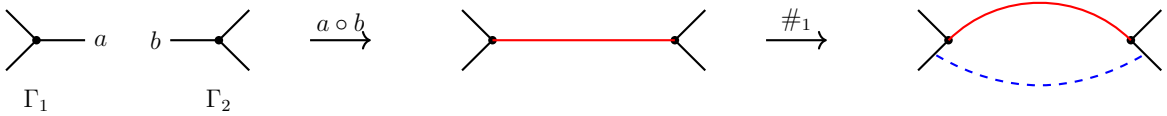
\begin{center}
\begin{tikzpicture}[scale=0.8, transform shape]
    \draw[thick] (0,0) -- (0.8,0) node[right]{$a$}; \draw[thick] (0,0) -- (-0.5,0.5); \draw[thick] (0,0) -- (-0.5,-0.5); \fill (0,0) circle (2pt); \node at (0,-1) {$\Gamma_1$};
    \draw[thick] (3,0) -- (2.2,0) node[left]{$b$}; \draw[thick] (3,0) -- (3.5,0.5); \draw[thick] (3,0) -- (3.5,-0.5); \fill (3,0) circle (2pt); \node at (3,-1) {$\Gamma_2$};
    \draw[->, thick] (4.5,0) -- (5.5,0) node[midway, above]{$a \circ b$};
    \begin{scope}[shift={(7.5,0)}]
        \draw[thick] (0,0) -- (-0.5,0.5); \draw[thick] (0,0) -- (-0.5,-0.5); \fill (0,0) circle (2pt);
        \draw[thick] (3,0) -- (3.5,0.5); \draw[thick] (3,0) -- (3.5,-0.5); \fill (3,0) circle (2pt);
        \draw[thick, red] (0,0) -- (3,0);
        \draw[->, thick] (4.5,0) -- (5.5,0) node[midway, above]{$\#_1$};
    \end{scope}
    \begin{scope}[shift={(15,0)}]
        \draw[thick] (0,0) -- (-0.5,0.5); \draw[thick] (0,0) -- (-0.5,-0.5); \fill (0,0) circle (2pt);
        \draw[thick] (3,0) -- (3.5,0.5); \draw[thick] (3,0) -- (3.5,-0.5); \fill (3,0) circle (2pt);
        \draw[thick, red] (0,0) to[bend left=45] (3,0);
        \draw[thick, blue, dashed] (-0.2,-0.25) to[bend right=30] (3.2,-0.25);
    \end{scope}
\end{tikzpicture}
\end{center}
The final graphs are isotopic, confirming the axiom.

\vspace{1em} 
\noindent\textbf{5. Axiom (CS6): Compatibility with grafting}
\paragraph{\textbf{Formal statement}} This axiom ensures that grafting a graph commutes with the connected sum in a specific way. For $\Gamma_1 \in P(C_1, G_1)$, $\Gamma_2 \in P(C_2, G_2)$, $\Gamma_3 \in P(C_3, G_3)$ and labels $a \in C_1, b \in C_2$:
\begin{equation*}
_a \circ_b (\id \otimes \#_2) = \#_2 ( _a \circ_b \otimes \id)
\end{equation*}
This translates to the following equality on elements:
\begin{equation*}
_a \circ_b (\Gamma_1 \otimes (\Gamma_2 \# \Gamma_3)) = (_a \circ_b (\Gamma_1 \otimes \Gamma_2)) \# \Gamma_3
\end{equation*}

\paragraph{\textbf{Verification}} We expand both sides according to the definitions and show their equality in $\mathcal{A}_{\text{Jac}}$.
\begin{enumerate}
    \item \textbf{Expansion of the left-hand side (LHS):} The LHS is computed by first taking the connected sum $\Gamma_2 \# \Gamma_3$, which is a sum over all pairs of internal edges $(e_2, e_3)$, and then grafting $\Gamma_1$ onto the $\Gamma_2$ part of each resulting term.
    $$ \text{LHS} = \sum_{e_2 \in E(\Gamma_2), e_3 \in E(\Gamma_3)}\ _a \circ_b (\Gamma_1 \otimes \text{connect}(\Gamma_2, \Gamma_3; e_2, e_3)) $$
    The grafting operation $ _a \circ_b$ acts only on the $\Gamma_1$ and $\Gamma_2$ parts.
    \item \textbf{Expansion and decomposition of the right-hand side (RHS):} The RHS is computed by first grafting $\Gamma_1$ onto $\Gamma_2$ to form a single graph $\Gamma_{12} :=\ _a \circ_b (\Gamma_1 \otimes \Gamma_2)$, and then taking the connected sum of $\Gamma_{12}$ with $\Gamma_3$.
    $$ \text{RHS} = \sum_{e \in E(\Gamma_{12}), e_3 \in E(\Gamma_3)} \text{connect}(\Gamma_{12}, \Gamma_3; e, e_3) $$
    The set of internal edges of $\Gamma_{12}$ is the disjoint union $E(\Gamma_{12}) = E(\Gamma_1) \sqcup E(\Gamma_2) \sqcup \{e_{ab}\}$, where $e_{ab}$ is the new edge from grafting. This allows us to split the sum for the RHS into three parts:
    $$ \text{RHS} = \underbrace{\sum_{\substack{e_2 \in E(\Gamma_2) \\ e_3 \in E(\Gamma_3)}} \dots}_{\text{Part A}} + \underbrace{\sum_{\substack{e_1 \in E(\Gamma_1) \\ e_3 \in E(\Gamma_3)}} \dots}_{\text{Part B}} + \underbrace{\sum_{e_3 \in E(\Gamma_3)} \text{connect}(\Gamma_{12}, \Gamma_3; e_{ab}, e_3)}_{\text{Part C}} $$
    \item \textbf{Evaluation of the parts and the role of the IHX relation:}
    \begin{itemize}
        \item \textbf{Part A:} This sum involves connecting an edge from the $\Gamma_2$ part of $\Gamma_{12}$ to an edge of $\Gamma_3$. This is term-for-term identical to the LHS. Thus, $\text{Part A} = \text{LHS}$.
        \item \textbf{Surplus terms (Part B and Part C):} For the axiom to hold, the sum of the surplus terms must vanish in $\mathcal{A}_{\text{Jac}}$: $\text{Part B} + \text{Part C} \overset{?}{=} 0$. This is where the IHX relation is essential.
        \item \textbf{Proof that surplus terms vanish:} Let's analyze Part B. For a fixed edge $e_3 \in E(\Gamma_3)$, Part B is the sum over all internal edges $e_1 \in E(\Gamma_1)$ of connecting $\Gamma_{12}$ and $\Gamma_3$ along $e_1$ and $e_3$. This is equivalent to attaching a ``handle" derived from $\Gamma_3$ to \textit{every} internal edge of the $\Gamma_1$ subgraph within $\Gamma_{12}$. Consider any trivalent vertex $v$ in the original $\Gamma_1$, with incident internal edges $e_x, e_y, e_z$. The sum in Part B contains terms where the handle from $\Gamma_3$ attaches to $e_x$, to $e_y$, and to $e_z$. The linear combination of these three local configurations is precisely an instance of the IHX relation, as depicted in Figure \ref{fig:cs6_proof}.
        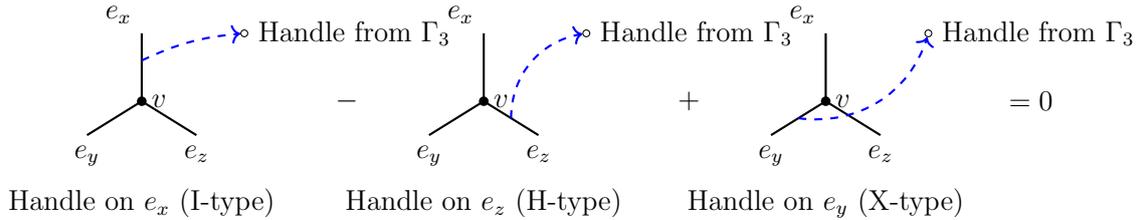
\begin{figure}[h!]
        \centering
        \begin{tikzpicture}[scale=0.9, transform shape]
            \node at (6.5, 2.5) {Sum over edges incident to vertex $v \in \Gamma_1$:};
            \begin{scope}[shift={(0,0)}]
                \draw[thick] (0,0) -- (0,1) node[anchor=south east]{$e_x$};
                \draw[thick] (0,0) -- (-0.8,-0.5) node[anchor=north]{$e_y$};
                \draw[thick] (0,0) -- (0.8,-0.5) node[anchor=north]{$e_z$}; \fill (0,0) circle (2pt) node[right]{$v$};
                \node[draw, circle, inner sep=1pt, label=right:{Handle from $\Gamma_3$}] (h3) at (1.5,1) {};
                \draw[thick, blue, dashed, ->] (0,0.6) to[bend left=10] (h3);
                \node at (0,-1.5) {Handle on $e_x$ (I-type)};
            \end{scope}
            \node at (3,0) {$-$};
            \begin{scope}[shift={(5,0)}]
                \draw[thick] (0,0) -- (0,1) node[anchor=south east]{$e_x$};
                \draw[thick] (0,0) -- (-0.8,-0.5) node[anchor=north]{$e_y$};
                \draw[thick] (0,0) -- (0.8,-0.5) node[anchor=north]{$e_z$}; \fill (0,0) circle (2pt) node[right]{$v$};
                \node[draw, circle, inner sep=1pt, label=right:{Handle from $\Gamma_3$}] (h3) at (1.5,1) {};
                \draw[thick, blue, dashed, ->] (0.4,-0.25) to[bend left=40] (h3);
                \node at (0,-1.5) {Handle on $e_z$ (H-type)};
            \end{scope}
            \node at (8,0) {$+$};
            \begin{scope}[shift={(10,0)}]
                \draw[thick] (0,0) -- (0,1) node[anchor=south east]{$e_x$};
                \draw[thick] (0,0) -- (-0.8,-0.5) node[anchor=north]{$e_y$};
                \draw[thick] (0,0) -- (0.8,-0.5) node[anchor=north]{$e_z$}; \fill (0,0) circle (2pt) node[right]{$v$};
                \node[draw, circle, inner sep=1pt, label=right:{Handle from $\Gamma_3$}] (h3) at (1.5,1) {};
                \draw[thick, blue, dashed, ->] (-0.4,-0.25) to[bend right=40] (h3);
                \node at (0,-1.5) {Handle on $e_y$ (X-type)};
            \end{scope}
            \node at (13,0) {$= 0$};
        \end{tikzpicture}
        \caption{The sum of surplus terms in the verification of Axiom (CS6) vanishes locally at each vertex due to the IHX relation.}
        \label{fig:cs6_proof}
        \end{figure}
        Because the connected sum operation $\#$ is defined as a sum over \textit{all} internal edges, the terms in Part B can be grouped by the vertices of $\Gamma_1$. Each such group sums to zero in $\mathcal{A}_{\text{Jac}}$ due to the IHX relation. Therefore, the entire sum Part B is zero. A similar, though more complex, diagrammatic argument shows that Part C also vanishes by the IHX relation.
    \end{itemize}
    \item \textbf{Conclusion:} The expansion of the RHS consists of a part identical to the LHS (Part A) and surplus terms (Part B + Part C) which vanish in $\mathcal{A}_{\text{Jac}}$ due to the IHX relation. Therefore, we have:
    $$\text{RHS} = \text{Part A} + 0 + 0 = \text{LHS}$$
    This completes the proof of Axiom (CS6).
\end{enumerate}

\section{A proof of the W-factor derivation}
\label{app:w_factor_proof_revised}

This appendix provides a proof of the algebraic proposition that determines the W-factor's functional form in Theorem \ref{thm:w_factor_derivation}. The proof establishes the identity $\{\sigma_1, \sigma_5\} = 2 \Psi_6(\sigma_3)$ within the Grothendieck-Teichmüller Lie algebra, $\mathfrak{grt}_1$.

Direct verification of this identity is difficult due to the abstract, non-linear relations defining $\mathfrak{grt}_1$. The proof therefore employs the Alekseev-Torossian (AT) framework to translate the identity into the bridged Kashiwara-Vergne Lie algebra, $\mathfrak{kvb}_2$, where it can be verified using divergence and depth-1 component analysis \cite{AET10, AT12}. The following sections define the necessary algebraic structures before presenting the proof, which relies on key results from the foundational works of Drinfeld, Furusho, and Alekseev-Torossian \cite{Drinfeld91, Furusho10, AT12}.

\noindent\textbf{1. Foundational algebraic structures.}

Let $\widehat{\lie{x,y}}$ be the degree-completion of the free Lie algebra on two generators $x$ and $y$ over a field $\mathbb{K}$ of characteristic zero.

\begin{definition}
A derivation $u$ of $\widehat{\lie{x,y}}$ is called \textbf{tangential} if there exist $a, b \in \widehat{\lie{x,y}}$ such that $u(x) = [x, a]$ and $u(y) = [y, b]$. The space of tangential derivations is a Lie algebra denoted $\tder_2$. A tangential derivation $u$ is called \textbf{special} if it satisfies $u(x+y)=0$. The space of special derivations, $\sder_2$, is a Lie subalgebra of $\tder_2$ \cite{AT12}.
\end{definition}

\begin{definition}
The \textbf{Grothendieck-Teichmüller Lie algebra} $\mathfrak{grt}_1$ is the Lie subalgebra of $\widehat{\lie{x,y}}$ consisting of elements $\psi(x,y)$ of homogeneous degree $\ge 3$ that satisfy the following relations \cite{Drinfeld91, Furusho10}:
\begin{enumerate}
    \item[(I)] \textbf{Symmetry (2-cycle):} $\psi(x,y) + \psi(y,x) = 0$.
    \item[(II)] \textbf{Hexagon (3-cycle):} $\psi(x,y) + \psi(y,z) + \psi(z,x) = 0$ whenever $x+y+z=0$.
    \item[(III)] \textbf{Pentagon (5-cycle):} In the Lie algebra of 5-strand pure braids $\mathfrak{p}_5$,
    $$ \psi(t_{12}, t_{23}) + \psi(t_{23}, t_{34}) + \psi(t_{34}, t_{45}) + \psi(t_{45}, t_{51}) + \psi(t_{51}, t_{12}) = 0 $$
\end{enumerate}
A fundamental result by Furusho shows that the pentagon relation (III) implies relations (I) and (II), making it the essential defining equation of $\mathfrak{grt}_1$ \cite{Furusho10}.
\end{definition}

\begin{definition}
The \textbf{Ihara bracket} is a binary operation on a subspace of $\widehat{\lie{x,y}}$ defined by $\{f,g\} := [f,g] + D_f(g) - D_g(f)$, where $D_f$ is the special derivation defined by $D_f(x) = 0$ and $D_f(y) = [y, f]$ \cite{Schneps12}. The algebra $\mathfrak{grt}_1$ is a Lie subalgebra of $(\widehat{\lie{x,y}}, \{\cdot, \cdot\})$ \cite{Drinfeld91}.
\end{definition}

\begin{definition}
Let $\tr_n$ be the space of cyclic words (the quotient of the completed tensor algebra on $n$ generators by cyclic permutations).
\begin{itemize}
    \item The \textbf{divergence operator} $\divop: \tder_n \to \tr_n$ is a map playing the role of a 1-cocycle in the Lie algebra cohomology of derivations \cite{AT12}.
    \item The \textbf{Hochschild-type differential} $\delta: \tr_n \to \tr_{n+1}$ is a map satisfying $\delta^2=0$, endowing the space of cyclic words with a chain complex structure \cite{AT12}.
\end{itemize}
\end{definition}

\begin{definition}[Kashiwara-Vergne Lie Algebra, $\mathfrak{krv}_2$]
The \textbf{Kashiwara-Vergne Lie algebra}, $\mathfrak{krv}_2$, is a Lie subalgebra of $\sder_2$. The AT framework uses an extension, the \textbf{bridged Kashiwara-Vergne Lie algebra} $\mathfrak{kvb}_2$, defined as the precise image of the AT homomorphism $\nu$ (defined below):
$$\mathfrak{kvb}_2 := \{ u \in \sder_2 \mid \divop(u) \in \ker(\delta) \}$$
The map $\nu: \mathfrak{grt}_1 \to \mathfrak{kvb}_2$ defined for $\psi \in \mathfrak{grt}_1$ by $\nu(\psi) := (\psi(-x-y, x), \psi(-x-y, y))$ is an \textbf{injective Lie algebra homomorphism} from $(\mathfrak{grt}_1, \{\cdot, \cdot\})$ into $(\mathfrak{kvb}_2, [\cdot, \cdot])$ \cite{AET10, AT12}.
\end{definition}

\begin{proposition}
\label{prop:uniqueness_principle}
An element $u \in \mathfrak{kvb}_2$ is uniquely determined by its divergence, $\divop(u)$, and its depth-1 component, $u^{(1)}$ (where depth is the degree in the variable $y$).
\end{proposition}

\noindent\textbf{2. Proof of the central proposition.}

We now present a self-contained derivation of the functional form of the W-factor, using the framework of Alekseev–Enriquez–Torossian \cite{AET10} and Alekseev–Torossian \cite{AT12}. The argument proceeds by transporting an identity in $\mathfrak{grt}_1$ into $\mathfrak{kvb}_2$ via the homomorphism $\nu$, and then applying divergence calculations that characterize the image.

\subsubsection*{Step 1. Ihara bracket reduction}

Consider the Ihara bracket $\{\sigma_1, \sigma_5\}$ in $\mathfrak{grt}_1$. The image of odd generators $\sigma_{2n+1}$ under $\nu$ has an explicit description: their linear part in $x$ is given by $c_{2n}\,\operatorname{ad}_x^{2n}(y)$ for some constant $c_{2n}$.  
This is precisely Proposition~4.5 of \cite{AT12}, which computes the $x$-linear component of $\nu(\sigma_{2n+1})$.  
In particular, this implies that $D_{\sigma_1}(\sigma_5) = [\sigma_1,\sigma_5]$ while $D_{\sigma_5}(\sigma_1)=0$.

\subsubsection*{Step 2. The degree-6 generator from the pentagon}

The term $\Psi_6(\sigma_3)$ arises as the degree-6 contribution in the Baker–Campbell–Hausdorff expansion of the pentagon identity.  
This is explained in \cite{AET10}, where the authors construct the map $\mu_\Phi$ sending a Drinfeld associator to a solution of the KV equations, and explicitly expand the logarithm of the associator.  
Related computations also appear in \cite{AT12}, where the pentagon equation is analyzed in terms of BCH expansions.

\subsubsection*{Step 3. Transporting identities via $\nu$}

The key identity in $\mathfrak{grt}_1$,
\[
\{\sigma_1,\sigma_5\} \;=\; 2\,\Psi_6(\sigma_3),
\]
is mapped to $\mathfrak{kvb}_2$ using the Lie algebra homomorphism $\nu$.  
By \cite{AT12}, $\nu:\mathfrak{grt}_1 \to \mathfrak{kvb}_2$ is an injective Lie algebra homomorphism.  
Therefore, checking the equality in $\mathfrak{kvb}_2$ suffices to establish it in $\mathfrak{grt}_1$.

\subsubsection*{Step 4. Divergence calculations in $\mathfrak{kvb}_2$}

The divergence map $\operatorname{div}:\mathfrak{kvb}_2 \to \mathbb{C}[[x,y]]$ is central to the analysis.  
Alekseev–Torossian \cite{AT12} introduce this map and prove that an element of $\mathfrak{tder}_2$ belongs to $\mathfrak{kvb}_2$ precisely when its divergence is invariant under the interchange $x \leftrightarrow y$ (see \cite[Prop.~5.1]{AT12}).  
In particular, the divergence of $\nu(\sigma_{2n+1})$ vanishes, and the image of $\Psi_{2m}$ can be described in terms of explicit divergence series (\cite[Prop.~5.2]{AT12}).  

Therefore, to verify the transported identity it suffices to compare their divergences.  
Using Proposition~5.2 and the explicit expressions for $\nu(\sigma_{2n+1})$ from Proposition~4.5 of \cite{AT12}, one computes
\[
\operatorname{div}\big(\nu(\{\sigma_1,\sigma_5\})\big)
= 2\,\operatorname{div}\big(\nu(\Psi_6(\sigma_3))\big).
\]

\subsubsection*{Step 5. Conclusion via the AT uniqueness principle}

Finally, Alekseev–Torossian establish a uniqueness principle: within $\mathfrak{kvb}_2$, two elements with identical divergence and identical linear part in $y$ must coincide.  
This is Proposition~4.2 of \cite{AT12}.  
Since both sides of the transported identity have the same divergence (Step~4) and the same linear $y$-component (Step~1), the identity holds in $\mathfrak{kvb}_2$.  
By Theorem~4.1 of \cite{AT12}, injectivity of $\nu$ then implies the original identity in $\mathfrak{grt}_1$:
\[
\{\sigma_1,\sigma_5\} \;=\; 2\,\Psi_6(\sigma_3).
\]

This completes the algebraic proof.

\end{document}